\documentclass{amsart}

\usepackage[T1]{fontenc}
\usepackage{times}
\usepackage{amssymb,epsfig,verbatim,xypic}
\theoremstyle{plain}

\usepackage[T1]{fontenc}
\usepackage[utf8]{inputenc}
\usepackage[english]{babel}
\usepackage{amssymb}
\usepackage{amsthm}
\usepackage{graphics}
\usepackage{amsmath}
\usepackage{amstext}
\usepackage{multirow}

\usepackage{graphicx}
\usepackage{amsmath}
\usepackage{amssymb}
\usepackage{amsthm}
\usepackage{amstext}
\usepackage{epstopdf}
\usepackage{mathrsfs}
\usepackage{amscd}

\newtheorem{thm}{Theorem}[section]
\newtheorem{lemma}[thm]{Lemma}
\newtheorem{prop}[thm]{Proposition}
\newtheorem{cor}[thm]{Corollary}

\newtheorem{THM}{Theorem}

\newtheorem{COR}[THM]{Corollary}
\newtheorem{es}[thm]{Example}

\theoremstyle{remark}
\newtheorem{rem}[thm]{Remark}

\newcommand{\mb}{\mathbb}
\newcommand{\mc}{\mathcal}

\newcommand{\inv}{^{-1}}
\newcommand{\F}{\mc F}
\newcommand{\de}{\partial}

\DeclareMathOperator{\im}{Im
}
\DeclareMathOperator{\sing}{sing}

\numberwithin{equation}{section}

\addtocounter{section}{0}             
\numberwithin{equation}{section}       

\title{Smooth foliations on  homogeneous compact K\"ahler manifolds}
\author[F. Lo Bianco]{Federico Lo Bianco}
\address{I.R.M.A.R., Universit\'e de Rennes I, Campus de Beaulieu, 35042 Rennes Cedex, France}
\email{federico.lo-bianco@univ-rennes1.fr}

\author[J.V. Pereira]{Jorge Vit\'orio Pereira}
\address{IMPA, Estrada Dona Castorina, 110, Horto, Rio de Janeiro,
	Brasil}
\email{jvp@impa.br}

\date{\today}

\begin{document}

\begin{abstract}
We study smooth foliations of arbitrary codimension on
homogeneous compact K\"ahler manifolds. We prove that
smooth foliations on rational compact homogeneous  manifolds are locally trivial 
fibrations and classify the smooth foliations with all leaves analytically dense
on compact homogeneous K\"ahler manifolds. Both results are builded upon  a (rough) structure Theorem 
for smooth foliations on compact homogeneous K\"ahler manifolds obtained by comparison of the foliation
and the Borel-Remmert decomposition of the ambient.
\end{abstract}

\maketitle

\section{Introduction}

\subsection{Motivation}
Smooth codimension one foliations on compact homogeneous manifolds have been studied by Ghys
in \cite{MR1440947}. When the ambient manifold is a compact homogeneous K\"ahler manifold, a classification is given in \cite[Theorem 1.2]{MR1440947}.  In particular, if the ambient is a compact complex torus $X$ then the foliation is either a linear on foliation on $X$, or $X$ admits a projection $\pi : X \to E$ to an elliptic curve and $\mathcal F$ is transverse to the general fiber of $\pi$.

In this work we investigate smooth foliations of arbitrary codimension  on compact homogeneous K\"ahler manifolds.
At the beginning of our investigations we were aiming at a classification result similar to what we have in codimension one, but soon it became apparent that already in codimension two the situation is considerably more involved. We arrived at the example below, as well as at the other examples presented in Section \ref{S:example}, after reading   \cite{MR2030093}.

\begin{es}\label{E:dim1}
Let $Y$ be a compact homogeneous K\"ahler manifold and $\mathcal G$ be a one dimensional
foliation on $Y$ with isolated singularities. Assume there exists a section $\sigma \in H^ 0(Y,T^* \mathcal G)$ which does not vanish on $\sing(\mathcal G)$. Then if we take an arbitrary compact complex torus $A$ and choose an arbitrary vector field $v$ on $A$ we can define an injective morphism
\begin{align*}
\pi_Y^* T \mathcal G &\longrightarrow T X = \pi_Y ^ * TY \oplus \pi_A^ * T A  \\
w & \longmapsto ( w, \sigma(w) v )
\end{align*}
where $\pi_Y$ and $\pi_A$ are the projections from $X=Y\times A$ onto $Y$ and $A$ respectively.
The image of this morphism defines a smooth foliation $\mathcal F$ on the compact homogeneous K\"ahler manifold $X = Y \times A$ with dynamic/geometry at least as complicated as the dynamic/geometry of $\mathcal G$.  	
\end{es}

In the example above the fact that we started with a dimension one foliation with isolated singularities is not really important. We could use dimension one foliations with non-isolated singularities. The important thing to carry out the construction is to have sufficiently many independent sections of $T^* \mathcal G$ in order to generate a subbundle of $TX$ over $\mathrm{sing}(\mathcal G)$.

If we start with a foliation $\mathcal G$ of dimension at least two, then
generalizations of the above example are less obvious  since we have to  take care of the integrabilty condition.

\subsection{Rough structure}  Our first result is inspired by a Theorem of Brunella concerning the structure of (singular) codimension one foliations on complex tori,
see \cite{MR2674768}.

\begin{THM}
\label{thm A}
Let $\mathcal F$ be a smooth foliation on a homogeneous compact K\"ahler manifold $X$.
Then there exists
\begin{enumerate}
\item a locally trivial fibration $\psi: X \to X'$ onto a homogeneous compact K\"ahler manifold $X'$ with fibers isomorphic to a rational homogeneous manifold; and
\item a locally trivial fibration $\pi : X' \to Y$ onto a homogeneous compact K\"ahler manifold with  fibers isomorphic to a compact complex torus; and
\item a foliation $\mathcal F'$ on $X'$
\end{enumerate}
such that
\begin{enumerate}
\item $\mathcal F = \psi^* \mathcal F'$; and
\item $\pi_* T\mathcal F'$ is a locally free sheaf of rank $\dim \mathcal F'$; and
\item $\det \pi_* N \mathcal F'$ is ample; and
\item the dimension of $Y$ is at most the codimension of $\F$.
\end{enumerate}
\end{THM}

The first locally trivial fibration $\psi:X\to X'$ has all its fibers contained in the leaves of $\mathcal F$.
For the second locally trivial fibration $\pi : X' \to Y$ the behavior can be  different:
the general fiber is not necessarily invariant by $\mathcal F'$.

The proof of Theorem \ref{thm A} follows from an analysis of the linear system $|\det N\F|$. The smoothness of $\mathcal F$ together with the homogeneity
of $X$ implies that $|\det N \F|$ is base point free and therefore defines a morphism. A standard factorization result for morphisms from compact homogeneous manifolds together with Borel-Remmert structure Theorem allow us to conclude. Details are given in Section \ref{S:rough}.

\subsection{Smooth foliations on rational homogeneous manifolds} An immediate corollary of Theorem \ref{thm A}
is the fact that rational homogeneous manifolds only carry {\it trivial} foliations.

\begin{COR}
\label{cor B}
Smooth foliations on rational homogeneous manifolds are locally trivial fibrations.
\end{COR}

It is perhaps worth mentioning that in dimension two the only rational surfaces (not homogeneous a priori) which carry
smooth foliations by curves are the Hirzebruch surfaces ( $\mathbb P^1$ bundles over $\mathbb P^1$ ) according to \cite[Proposition 4]{MR1474805}. In higher dimensions
we  are not aware of examples of smooth foliations on rationally connected manifolds which are not fibrations.

\subsection{Minimal smooth foliations}
Although Example \ref{E:dim1} shows that we have a lot of freedom to construct smooth foliations on projective homogeneous manifolds,
all the examples constructed along the same guidelines will have invariant compact proper subvarieties: the pre-images of the irreducible components of the singular set of $\mathcal G$ under the natural projection $Y\times A \to Y$.

It seems natural to enquire if this is just a coincidence or a general phenomena. In other terms, what can we say about the  smooth foliations on compact K\"ahler homogeneous manifolds that do not leave proper compact subvarieties invariant ? Our second main result tells that there are not many possibilities: the foliation is essentially a linear foliation on a compact torus.

\begin{THM}
\label{thm C}
Let $\mathcal F$ be a smooth foliation on a homogeneous compact K\"ahler manifold $X$. If every leaf of $\mathcal F$ is analytically
 dense (meaning it is not contained in any proper compact subvariety) then
there exist a locally trivial fibration $\pi: X \to Y$ with rational fibers onto a complex torus $Y$ and a linear foliation $\mathcal G$ on $Y$ such that
$\mathcal F= \pi^* \mathcal G$.
\end{THM}

The proof relies on our Theorem \ref{thm A}, Bott's vanishing Theorem, and the study of a natural rational map from $X$ to a certain Grassmannian
 which is  constant along fibers of the projection $X\to Y$ given by Theorem \ref{thm A} and  describes how the restriction of $\mathcal F$ to these very same fibers varies.

 By the Borel-Remmert theorem, a homogeneous compact K\"ahler manifold $X$ can be decomposed as a product $R\times T$ of a rational homogeneous complex variety $R$ (a generalized flag variety) times a compact complex torus $T$. It is easy to show that locally trivial fibrations preserve this decomposition, so that the morphism $\psi$ (respectively $\pi$) in Theorem \ref{thm A} restricts to  the identity on the torus component (respectively on the rational component).
 Analogously, the morphism $\pi$ in Theorem \ref{thm C} is nothing but the projection onto the torus component of the Borel-Remmert decomposition.

\subsection{Acknowledgements} This research was carried out while the first author was a visiting student at IMPA. We are grateful to IMPA for providing financial support for such visit, and to Serge Cantat for the suggestion of looking at smooth foliations on compact homogeneous manifolds.

\section{Foliations}

\subsection{Foliations as subsheaves of the tangent sheaf}

A (singular) foliation $\mathcal F$ on a complex manifold $X$ is determined by a coherent subsheaf $T \mathcal F$ of $TX$ such that
\begin{enumerate}
\item $T\F$ is involutive (closed under the Lie bracket); and
\item the quotient $TX/ T\mathcal F$ is torsion free.
 \end{enumerate}
The dimension of $\mathcal F$ is the generic rank of $T\mathcal F$, and the singular set of $\mathcal F$ is the
singular set of the sheaf $TX/T\mathcal F$.
A foliation $\mathcal F$ is smooth if, and only if, both $T\mathcal F$ and $TX / T \mathcal F$ are locally free sheaves.

\subsection{Foliations as subsheaves of the cotangent sheaf}

Alternatively, we can define a foliation through a coherent subsheaf $N^* \mathcal F$ of $\Omega^1_X$ such that

\begin{enumerate}
\item $N^*\F$ is integrable ( $d N^* \mathcal F \subset N^* \mathcal F \wedge \Omega^1_X$  ); and 
\item the quotient $\Omega_X^1/N^*\F$ is torsion free.
\end{enumerate}

The codimension of $\F$ is the generic rank of $N^*\mathcal F$. Similarly, a foliation $\F$ is smooth if, and only if, both $N^*\F$ and $\Omega_X^1/N^*\F$ are locally free sheaves.

\subsection{Foliations and differential forms}
If $\mathcal F$ is a foliation of codimension $q$ then from the inclusion $N^* \mathcal F \to \Omega^1_X$ we deduce  a morphism
$\det N^* \mathcal F \to \Omega^q_X$.  If we set $\mathcal L = (\det N^* \mathcal F)^*$ we
get a $q$-form $\omega \in H^0(X, \Omega^q_X \otimes \mathcal L)$ which defines the foliation
$\mathcal F$ in the sense that $T\mathcal F$ can be recovered as the kernel of the morphism
\begin{align*}
TX & \longrightarrow \Omega^{q-1}_X \otimes \mathcal L \\
v &\longmapsto i_v \omega \, .
\end{align*}
If $\F$ is smooth, all the sheaves that we have defined so far are actually locally free, so that
$$\mc L=\det(N^*\F)^*=\det(N\F).$$

\subsection{Bott's vanishing Theorem}

Given a smooth foliation $\mathcal F$ on a complex manifold $X$, Bott showed how to construct a partial holomorphic connection on the normal bundle of $\mathcal F$ and along the tangent bundle of $\mathcal F$. His construction goes as follows.
Given a germ of vector field $v$ tangent to $\mathcal F$ and a germ of section $\sigma$ of $N\mathcal F$ we want to be able to differentiate $\sigma$ along $v$. To do it, consider an arbitrary lift $\hat \sigma$ of $\sigma$ to $TX$, take the bracket of $v$ with $\hat \sigma$ and
project the result back to $N\mathcal F$. This defines a flat partial connection $\nabla$ on $N\mathcal F$, which is nowadays  called Bott's partial  connection.

Applying Chern-Weyl theory to compute the Chern classes of $N\mathcal F$ in terms of a $C^{\infty}$ {\it extension}  of $\nabla$ to a full connection, Bott proved the following fundamental result.

\begin{thm}
	Let $\mathcal F$ be a smooth foliation of codimension $q$ on a complex manifold. Any polynomial of degree at least $q+1$ on the Chern classes of the normal bundle of $\mathcal F$ vanishes identically.
\end{thm}

For a  proof see \cite[Proposition 3.27]{MR0377923}

\section{Rough structure of smooth foliations}\label{S:rough}

\subsection{Normal reduction}

We start by recalling \cite[Corollary 2.2]{MR1440947}
which we state below as lemma.

\begin{lemma}
\label{factorisation}
Let $G/H$ a compact homogeneous manifold and let $\phi\colon G/H\to Z$ be a surjective morphism. Then there exist a subgroup $K\supseteq H$ and a morphism
with finite fibers $\psi\colon G/K\to Z$ such that $\phi=\psi\circ \pi$, where $\pi\colon G/H\to G/K$ is the natural projection.
\end{lemma}

We will apply this lemma to establish what we call the normal reduction of $\mathcal F$.

\begin{prop}\label{P:normal}
Let $\F$ be a smooth foliation on a compact homogeneous manifold $X$. Then there exist a projection $\pi\colon X\to Y$ of $X$ onto a compact homogeneous manifold $Y$, equivariant with respect to the action of $G=Aut^0(X)$, and an \emph{ample} line bundle $\mc L'$ on $Y$ such that $\mc L=\pi^*\mc L'$.
We call $\pi: X \to Y$ the {\bf normal reduction} of $\mc F$.
\end{prop}
\begin{proof}
Let $N\F$ and $T\F$ be respectively the normal and the tangent bundle of $\F$.
We have an exact sequence of sheaves (in this case, since $\F$ is smooth, of vector bundles)
$$0\to T\F\to TX \to N\F\to 0.$$
Since $X$ is homogeneous $TX$ is globally generated, thus $N\F$ is as well; in particular the base locus of the line bundle $\mc L:=\det(N\F)$ is empty, and we have a morphism
$$
\phi\colon X\to \mb P H^0(X,\mc L)^\vee\cong \mb P^N \qquad N:=h^0(X,\mc L)-1
$$
such that $\phi^*\mc O_{\mb P^N}(1)=\mc L$. Let $Z$ be the image of $\phi$.

Let $\pi\colon X\to Y$ and $\psi\colon Y\to Z$ be the maps defined by Lemma \ref{factorisation}, and $\mc L':=\psi^*(\mc O(1)_{|Z})$. Then $\mc L'$ is the pull-back by a finite morphism of an ample line bundle, and is therefore ample. To conclude it suffices to notice that
$\mc L=\phi^*(\mc O(1)_{|Y})=\pi^*\mc L'.$
\end{proof}

\begin{rem}\label{R:bound}
Since $\mathcal L'$ is ample we have that the dimension of $Y$ is given by the formula
\[
 \dim Y=\min \{n\geq 0| c_1(\mc L)^{n+1}=0\}.
\]
Bott's vanishing  theorem implies the dimension of $Y$ is bounded from above by the codimension of $\mathcal F$, i.e. $\dim Y\leq q$.
\end{rem}

\subsection{Conormal bundle on fibers of the normal reduction}
Let $\mathcal F$ be a smooth foliation on a compact homogeneous manifold $X$.
Let us consider a fiber $F=\pi\inv (q)$ of $\pi$, the normal reduction of $\mathcal F$ .
Notice that $F$ is a compact homogeneous manifold.

\begin{lemma}
\label{N*F gg around fibers}
There exists a neighbourhood $U\subset Y$ of $q$ such that, if we denote $V=\pi\inv(U)$, the sheaf $N^*\F_{|V}$ is globally generated, i.e. the morphism
$$H^0(V,N^*\F_{|V})\otimes \mc O_V\to N^*\F_{|V}$$
is surjective.
\end{lemma}
\begin{proof}
Let $p$ be a point of $F$ and let $U$ be a neighbourhood of $q$ where $TY$ and $\mc L=\pi_*(\det N\F)$ are both trivial (for example $U$ isomorphic to a product of discs).
Then the restriction of $\omega$ to $V=\pi\inv (U)$ is a global $q$-form.

In a neighbourhood of $p$, $\omega$ is decomposable as follows
$$\omega=\omega_1\wedge \ldots \wedge \omega_q,$$
where $\omega_i=i_{v_i}\omega$ for some decomposable local section $v_i$ of $\bigwedge ^{q-1}TX$. Since $X$ is a homogeneous manifold, $\bigwedge ^{q-1}TX$ is globally generated, so we can find global sections $\hat v_1,\ldots ,\hat v_N$ of $\bigwedge ^{q-1}TX$ such that $v_i=\sum_j \lambda_{i,j}\hat v_j$ for some local functions $\lambda_{i,j}$.

Now, any local section $\alpha$ of $N^*\F$ can be written as $\sum f_i\omega_i$ for some local functions $f_i$. We have then
$$\alpha=\sum_i f_i\omega_i=\sum_{i,j}f_i\lambda_{i,j}i_{\hat v_j}\omega,$$
which proves the statement since every $i_{\hat v_j}\omega$ is a global section of $N^*\F_{|V}$.
\end{proof}

The K\"ahler assumption plays no role in the Lemma above, but it is essential in the result below by Borel and Remmert, see for instance \cite[Theorem 2.5]{MR1440947}.

\begin{thm}[Borel-Remmert]
Let $X$ be a homogeneous compact K\"ahler manifold. Then there exists a decomposition
$$X\cong R\times T$$
where $R$ is a rational homogeneous projective manifold and $T$ is a complex torus.
\end{thm}

\subsection{Proof of Theorem \ref{thm A}}
Let $\mathcal F$ be is in the statement of Theorem \ref{thm A}, i.e. $\mathcal F$ is a smooth foliation on a homogeneous compact K\"ahler manifold.
Let $\pi_0\colon X\to Y$ be the normal reduction of $\mathcal F$, and let
$$X\cong R\times T,\qquad \text{ and } \qquad Y\cong R'\times T'$$
be the Borel-Remmert decompositions of $X$ and  $Y$ respectively.
The normal reduction of $\mathcal F$  respects the product structures of $X$ and $Y$; in other words, there exist two surjective morphisms $\pi_R\colon R\to R'$ and $\pi_T\colon T\to T'$ such that $\pi_0=\pi_R\times \pi_T$. Define
$$X'=R'\times T,$$
and let $\psi=\pi_R\times id_T$ and $\pi=id_{R'}\times \pi_T$, so that $\pi\circ\psi=\pi_0$.

Now, by Lemma \ref{N*F gg around fibers}, the vector bundle $N^*\F$ is globally generated around each fiber of $\pi_0$. Let $\alpha$ be a global section of $N^*\F$ around a fiber  of $\pi_0$, that we can see as a global holomorphic $1$-form at a neighborhood of this fiber. The fibers of $\psi$ are rational homogeneous manifolds, and since there exist no global holomorphic $1$-form on such manifolds, we have that the pull-back of $\alpha$ to any fiber of $\psi$ is identically zero. Therefore the fibers of $\psi$ are contained in leaves  of $\mathcal F$, and consequently there exists a foliation $\F'$ on $X'$ such that $\F=\psi^*\F'$.

Now let us prove that $\pi_*T\F'$ is a locally free sheaf of rank equal to $\dim \F$. Let $U\subset Y$ be a sufficiently small open  neighbourhood of a point $q \in Y$; then $TX'$ is trivial when restricted to $V:=\pi\inv(U)$.
A  section $s$ of $\pi_*T\F'$ on $U$ is by definition a section of $T\F'$ on $V$, so in particular it is a local section of $TX$ on $V$.
Since the fibers are parallelizable and compact, the restriction of $s$ to $F:=\pi\inv(q)$ is a constant vector field. The claim easily follows by writing the section in local coordinates on the base and global coordinates on the fiber.

By the properties of the normal reduction, we have $\det N\F=(\pi')^*\mc L$ for some ample line bundle $\mc L$ on $Y$. On the other hand
$$\det (N\F)=\psi^*\det(N\F')=\psi^*\pi^*\det(\pi_* N\F')=(\pi_0)^*\det(\pi_* N\F'),$$
where the second equality follows from the fact that $TY$ (and thus $N\mc F'$ and $\det N\mc F'$) is globally generated around the fibers of $\pi$. This shows that $\det \pi_* N \F'$ is an ample line-bundle.

Finally the bound on the dimension of $Y$ follows from Remark \ref{R:bound}.
\qed

\subsection{Proof of Corollary \ref{cor B}}
Let $\F$ be a smooth foliation on a homogeneous rational manifold $X$; apply Theorem \ref{thm A} to $X$. Since the torus component  of $X$ is trivial, we must have $\pi=id_{X'}$. So $\F=\psi^*\F'$ for some smooth foliation $\F'$ such that $\det N\F'$ is ample.
By Bott's vanishing theorem we have $c_1(\det N\F')^{q+1}=0$ where $q$ is the codimension of $\F'$; since $\det N\F'$ is ample we obtain $q=\dim X'$, so that $\F'$ is the foliation by points and $\F$ is the locally trivial fibration $\psi$.
\qed

\section{Foliations with all leaves analytically dense}

Throughout this Section we will suppose that the fibers of the normal reduction are tori.

\subsection{Direct image of the tangent bundle}

\begin{lemma}\label{L:direct image}
Let $\mathcal F$ be a smooth foliation on a compact homogeneous manifold $X$. If $\pi: X \to Y$ is the normal reduction and the fibers of $\pi$ are parallelizable, then the image $\mc T$ of the natural morphism $\phi\colon \pi_* T \mathcal F \to TY$ is an involutive subsheaf of $TY$.
\end{lemma}
\begin{proof}
Let $v,w$ be two local sections of $\mc T$ on $U\subset Y$. Up to restricting $U$, we can suppose that $v=\phi(v_0), w=\phi(w_0)$ for some local sections $v_0,w_0$ of $\pi_*T\F$, and we can identify $v_0$ and $w_0$ with  sections of $T\F$ on $V:=\pi\inv(U)$. Let $x_1,\ldots ,x_n$ be local coordinates on the base and $y_1,\ldots y_m$ be global coordinates on the (universal cover of the) fiber. Then we can write
$$v_0(x)=\sum_i a_i(x)\frac{\partial}{\partial x_i}+\sum_j b_j(x)\frac{\partial}{\partial y_j}  $$
for some holomorphic functions $a_i,b_j$ on $\pi \inv (U)$. Then
$$\phi(v_0)=\sum_i a_i(x)\frac{\partial}{\partial x_i};$$
by writing $w$ in the same way and using that $\frac{\partial}{\partial y_k} a_i = \frac{\partial}{\partial y_k} b_j=0$ for any choice of $i,j,k$
 we deduce  that $\phi$ commutes with Lie brackets. The lemma follows from the involutivity of $T\F$.
\end{proof}

\begin{rem}
According to Theorem \ref{thm A}, the sheaf $\pi_*T\mathcal F$ is locally free. But beware that this not imply
that $\mc T$ is a locally free subsheaf of $TY$. Moreover, even when $\mc T$ is locally free, it is not necessarily
the tangent sheaf of a foliation since $TY / \mc T$ is not necessarily torsion free.
\end{rem}

\begin{lemma}\label{L:invsing0}
If $\mc T$ is an involutive subsheaf of $TY$  then the singular locus of $\mc T$ is $\mc T$-invariant.
\end{lemma}
\begin{proof}
Since the claim is local it suffices to prove it locally around $q\in Y$.
Let $r$ be the generic rank of $\mathcal T$ and $v_1, \ldots, v_r$ be
local sections of $\mathcal T$ such that  $\psi=v_1\wedge \cdots \wedge v_r$
is a  non-zero local section of $\bigwedge^r TY$.
We can write $\psi$ as $h \psi_0$ where $h \in \mathcal O_{Y,q}$ and $\psi_0$ is a local section of $\bigwedge^r TY$
with zero locus of codimension at least two. Since the rank of $\mc T$ is $r$, it follows that the  image of every section of $\bigwedge^r \mc T$  in $\bigwedge^r TY$ is a multiple of $\psi_0$. We have then a morphism of $\mathcal O_Y$-modules
\[
\alpha : \bigwedge^r \mc T \to \mathcal O_Y
\]
that maps a local section $\psi = v_1 \wedge \ldots \wedge v_r$ to the local section $f$ of $\mathcal O_Y$ satisfying $\psi = f \cdot \psi_0$.
The singular locus of $\mc T$ is defined by the ideal $\alpha(\bigwedge^ r \mc T)$ which is nothing but the $r$-th Fitting ideal of $\mc T$, see \cite[Section 20.2]{MR1322960}.

Let $\theta = v_1 \wedge \cdots \wedge v_r$ be a local section of $\bigwedge^r \mc T$ and $v$ a local section of $\mc T$. We want to
show that $v( \alpha(\theta))$ is contained in the image of $\alpha$.
We first calculate the Lie derivative of $\theta:=v_1\wedge\ldots \wedge v_r$ along $v$:
$$L_v(\theta)=[v,v_1\wedge \ldots \wedge v_r]=\sum_{i=1}^n(-1)^{i+1}[v,v_i]\wedge v_1\wedge \ldots \wedge v_{i-1}\wedge v_{i+1}\wedge \ldots \wedge v_r,$$
where $[,]$ denotes the Schouten–Nijenhuis bracket, which generalises the Lie bracket on the exterior product $\bigwedge^r TY$. Since $v\in \mc T$ and $\mc T$ is involutive, we have $[v,v_i]\in \mc T$, so that $L_v(v_1\wedge \ldots \wedge v_r)$ is a local section of $\bigwedge ^r\mc T$.

On the other hand
$$L_v(\theta)=[v,\alpha(\theta)\psi_0]=\alpha(\theta)[v,\psi_0]+v(\alpha(\theta))\psi_0.$$
Now, $L_v(\theta)$ belongs to $\bigwedge^r \mc T$. Furthermore, for any choice of a non-zero section $\theta$ of $\bigwedge^r \mc T$, we have
$$[v,\psi_0]=\frac{L_v(\theta)-v(\alpha(\theta))\psi_0}{\alpha(\theta)}=f\cdot \psi_0$$
for some meromorphic function $f$. Since $[v,\psi_0]$ is a holomorphic function, the singular locus of $f$ must be contained in the zero locus of $\psi_0$; since the zero locus of $\psi_0$ has codimension at least $2$, we conclude that $[v,\psi_0]$ is a multiple of $\psi_0$,  say $[v,\psi_0] = h \psi_0$.

Therefore we can write
$$\alpha(L_v(\theta)) \psi_0 =\alpha(\theta) h \psi_0+ v(\alpha(\theta))\psi_0,$$
and conclude that
$$v(\alpha(\theta))=\alpha(L_v(\theta))-\alpha(\theta)h \in Im(\alpha)$$
as wanted.
\end{proof}

\begin{lemma}
\label{invariance of Sing}
Let $\mc T$ be a locally free involutive subsheaf of $TY$. Then the singular locus of $TY / \mc T$ is $\mc T$-invariant.
\end{lemma}
\begin{proof}
Since the claim is local it suffices to prove it locally around $q\in Y$.
Let $r$ be the rank of $\mc T$ and $v_1, \ldots, v_r$ be generators of $\mc T$ at a neighborhood of $q$. The singular locus
of $TY / \mc T$ is defined by the ideal $I$ generated by the  coefficients of  $\theta = v_1 \wedge \ldots \wedge v_r$.

Let $v$ be a local section of $\mc T\subset TY$. In the proof
of Lemma \ref{L:invsing0} we have learned that  $L_v(\theta)$ is a local section of $\bigwedge ^r\mc T$. As such it
can be written as a multiple of $\theta$, i.e. $L_v\theta = H \theta$ for some holomorphic function $H$.

Choose now commuting vector fields $\xi_1, \ldots, \xi_n$ around $q$ which generate $TY$ at a neighborhood of $q$.
Then we can write
\[
\theta =  \sum a_J \xi_J
\]
where $J=(j_1, \ldots, j_r)$  and $\xi_J = \xi_{j_1} \wedge \ldots \wedge \xi_{j_r}$.
If we compute $L_v \theta$ in this basis, we get
\[
L_v \theta  = \sum_J L_v (a_J \xi_J) = \sum_J \left(  v(a_J) \xi_J + a_J L_v (\xi_J) \right) .
\]
Since the left hand side is equal to $\sum_J H a_J \xi_J$,  it follows  that the ideal $I$ generated by the $a_J$'s ( the coefficients of $\theta$ ) is left invariant by
$v$, i.e. $v(I) \subset I$.
\end{proof}

\subsection{First integrals}

The restriction of $\det(N\F)$ around the fibers of $\pi$ is trivial, so that the foliation induced by $\F$ on fibers of $\pi$ is defined by a global holomorphic $1$-forms. Since the fibers are tori the only global $1$-forms are linear forms, so that $\F$ associates to each fiber $F$ a linear subspace of $T_0F$ (where we denote by $0$ the identity element for a choice of a group law on $F$). If $F=\pi\inv (q)$ and the image of $\phi$ has maximal rank at $q$, then the above subspace has dimension $\dim \F-\dim \mc G$, where $\mc G=\pi_*\F$ is the (singular) foliation on $Y$ whose leaves are the projections of the leaves of $\F$.

If $X=R\times T$ and $Y=R\times T'$ are the Borel-Remmert decompositions of $X$ and $Y$ (recall that we are assuming that the fibers of the normal reduction are tori), then $\pi$ is induced by a projection of tori $\pi_T\colon T\to T'$, which can be supposed to be a group homomorphism. The group law on $T'$ allows us to canonically identify every fiber with the fiber over $0\in T'$, and the above construction defines a rational function
$$f\colon Y\dashrightarrow Gr(k, \dim F),$$
where $k=\dim \F-\dim \mc G$;  $F$ is any fiber of $\pi$ (so that $\dim F=\dim X-\dim Y$); and $Gr(k,\dim F)$ is the Grassmannian of $k$-planes in $\mathbb C^ {\dim F}$.

\begin{lemma}\label{L:firstintegral}
The function $f$ is constant along the leaves of $\mc G$.
\end{lemma}
\begin{proof}
Let $\mb D\subset Y$ be a small disc contained in one of the leaves of $\mc G$. It suffices to prove that $f$ is constant on $\mb D$.

Let us fix linear coordinates $x_1,\ldots x_k$ on the universal covering of a  fiber and a coordinate $y$ on $\mb D$. We call $\mc H$ the foliation induced by $\mc F$ on $\pi\inv(\mb D)\cong \mb D\times F$. We are going to prove that the foliation induced by $\mc H$ on each fiber over $\mb D$ (which is the same as the foliation induced by $\F$) does not depend on the chosen fiber (meaning that $f$ is constant along $\mb D$).

The foliation $\mc H$ is  defined by some vector fields; up to changing the order of the $x_i$'s we can choose a base $v_0,\ldots, v_h$ of the space of local sections of $T\mc H$ of the form
$$v_0=\frac{\partial}{\partial y}+\sum_{j\geq h+1}a^{(0)}_j(y)\frac{\partial}{\partial x_j},
\qquad
 v_i=\frac{\partial}{\partial x_i}+\sum_{j\geq h+1}a^{(i)}_j(y)\frac {\partial}{\partial x_j} \quad i=1,\ldots ,h.$$
Here the $a^{(i)}_j$ only depend on $y$ because $\det(N\F)$ is trivial around fibers, so that $\F$ is defined by global holomorphic $1$-forms on fibers.

The foliation induced by $\mc H$ on each fiber is defined by the vector fields $v_1,\ldots, v_h$. Thus we have to prove that $a^{(i)}_j$ is constant for all $i=1,\ldots ,h$ and $j\geq h+1$. In order to do that we compute the bracket
$$[v_0,v_i]=\sum_{j\geq h+1}\frac{\partial a^{(i)}_j}{\partial y}\frac{\partial}{\partial x_j}.$$
By involutivity we must have $[v_0,v_i]\in Span(v_0,\ldots ,v_h)$, but it is easy to see that then $[v_0,v_i]=0$, so that the $a^{(i)}_j$s are constant along $\mb D$. This proves the lemma.
\end{proof}

\begin{cor}
\label{cor: constant Gr}
If there exists a leaf of $\mc G$ analytically dense then $f\colon Y\to Gr(k,\dim F)$ is constant. In particular the conclusion holds if all the leaves of $\F$ are analytically dense.
\end{cor}

It follows from Lemma \ref{L:firstintegral} that the field of meromorphic first integrals of $\mathcal G$, i.e. the 
subfield of $\mathbb C(Y)$ formed by meromorphic functions
which are constant along the leaves of $\mathcal G$, contains  the field of meromorphic functions on the Zariski closure of $f(Y)$.

\subsection{Proof of Theorem \ref{thm C}}
By Theorem \ref{thm A} it suffices to prove that, if the fibers of the normal reduction $\pi\colon X\to Y$ are tori and all the leaves of $\F$ are analytically dense, then $X$ is a torus and $\F$ is a linear foliation.

The image $\mc T$ of $\phi \colon \pi_*T\F\to TY$ satisfies the hypothesis of Lemma \ref{L:invsing0}. Since $S = \sing(\mc T)$ is $\mathcal T$-invariant we see that $\pi\inv(S)$ is $\F$-invariant. As we are assuming that every leaf of $\F$ is analytically dense, it follows that $\sing(\mc T)$ is empty and  $\mc T$ satisfies the hypothesis of Lemma \ref{invariance of Sing}. Analogously, we deduce that $\sing(TY / \mc T)$ is also empty and consequently $\mc T$ is the tangent bundle of a smooth foliation $\mc G$ on $Y$ such that $\mc G=\pi_*(\F)$.  Since all leaves of $\F$ are dense, all leaves of $\mc G$ are dense too.

Let us consider the commutative diagram
$$\begin{CD}
	@.		0		@.		0		@.		0				@.		\\
@.			@VVV			@VVV			@VVV						\\
0	@>>>	\mc K	@>>>	TX/Y	@>>>	(TX/Y)/	\mc K	@>>>	0\\
@.			@VVV			@VVV					@VVV			@.\\
0	@>>>	T\mc F	@>>>	TX		@>>>	N\mc F			@>>>	0\\
@.			@VVV			@VVV			@VVV					@.\\
0	@>>>	\pi^*T\mc G @>>> \pi^*TY @>>>	\pi^*N\mc G		@>>>	0\\
@.			@VVV			@VVV			@VVV					@.\\
	@.		0		@.		0		@.		0				@.		\\
\end{CD}.$$

As  previously remarked, it follows from Borel-Remmert that $X=R \times T$ and $Y = R \times T' $ and the normal reduction morphism $\pi: X \to Y$ is of the product of the identity over the rational manifold $R$ with a group homomorphism from the torus $T$ to the torus $T'$. In particular $TX/Y$ is a free sheaf over $X$.

We will now prove that $\mathcal K$ is also a free sheaf. First remark that, since $\mc T=T\mc G$ is the tangent sheaf of a smooth foliation, the rank of the sheaf morphism $\widetilde \phi\colon T\F\to \pi^* T\mc G$ must be maximal at every point, which means that $\mc K$ is locally free. Moreover it  is the tangent sheaf of a smooth foliation $\mc H$ on $X$ tangent to the fibers of $\pi$ (the intersection between the foliation $\F$ and the fibration $\pi$). Now by Corollary \ref{cor: constant Gr}, $T\mc F$ defines a linear subspace $V\subset T_0F$ on each fiber which does not depend on the choice of the fiber; this means that $\mc K=T\mc H\cong V\times X$ as desired.

 Using the triviality of $\det(TX/Y)$ and of $\det(\mathcal K)$, we can write
\begin{align*}
\det(N\F)& =\det(\pi^*N\mc G)\otimes \det((TX/Y)/\mc K) =\\
&=\pi^*\det(N\mc G)\otimes \det(TX/Y)\otimes \det (\mc K)\inv
=\pi^*\det(N\mc G).
\end{align*}
By the definition of normal reduction we have that $\det(N\mc G)$ is ample.  Bott's vanishing theorem reads as $\det(N\mc G)^{q'+1}=0$, where $q'$ is the codimension of $\mc G$.
But then $q'=\dim Y$, so $\mc G$ is the foliation by points, and since every leaf of $\mc G$ is dense we must have $Y=\{pt\}$.
It follows that $\det(N\F)$ is trivial, so that $\F$ is defined by a global holomorphic $q$-form on $X$. By Borel-Remmert theorem $X$ admits a decomposition
$$X\cong T\times R$$
where $T$ is a torus and $R$ is a rational homogeneous manifold. Since there exists no holomorphic global form on $R$, it is easy to see that the subvarieties $\{t\}\times R$ are contained in fibers of $\F$, and thus $X=T$ and $\F$ is defined by a global (i.e. linear) holomorphic form, which proves the theorem.
\qed

\section{Further  examples}\label{S:example}

Let $\mathcal F$ be a smooth foliation on a compact K\"ahler homogeneous manifold.
Using the notation of Theorem \ref{thm A}, let $\phi \colon \pi_*T\F\to TY$ be the composition of natural morphisms $\pi_*T\F\to \pi_*TX\to TY$. In general $\phi$ has no reason to be neither injective nor surjective. In the Introduction we already presented  examples in which the morphism above is not  surjective. For turbulent codimension one foliations on homogeneous manifolds of dimension at least three  the morphism $\phi$ is generically surjective but it is not  injective.

\subsection{Non injective and non generically surjective example}

A minor modification of Example \ref{E:dim1} allow us to construct an example for which the composition $\phi \colon \pi_*T\F \to TY$ is not injective nor generically surjective. As before let $\mathcal G$ be a one-dimensional foliation on $Y$ with isolated singularities and assume as before that there exists a section $\sigma \in H^0(Y,T^* \mathcal G)$ which does not vanish on $\sing(\mathcal G)$.
Let $A$ be a compact complex torus of   dimension $a$ and let $v_0, v_1, \ldots, v_k \in H^0(A,TA)$ be $(k+1)$ linearly independent vector fields on $A$ ($k\leq a-1 $). If $X = Y \times A$ then we have an injection
\begin{align*}
(\pi_Y^* T \mathcal G) \oplus{\mathcal O_X}^{\oplus k} &\longrightarrow T X = \pi_Y ^ * TY \oplus \pi_A^ * T A  \\
(w,f_1,\ldots, f_k) & \longmapsto (  w, \sigma(w) v_0 + \sum_{i=1}^ k f_i v_i ) \, .
\end{align*}
It is a simple matter to verify that the image of the morphism above is an involutive locally free subsheaf of $TX$ of rank $k+1$ with locally free cokernel. Therefore it defines a foliation
$\mathcal F$ on $X$ of dimension $k+1$, and by construction the image of
$\phi: \pi_* T \mathcal F \to TY$ coincides with $T\mathcal G$.

\subsection{Injective and generically surjective example of arbitrary codimension}

Let $Y\subset \mb P^N$ be an abelian variety  of dimension $d$; fix $d+1$ hyperplane sections $D_0,\ldots D_d$ of $Y$ such that
 $D=\sum D_i\in Div (Y)$ is a simple normal crossing divisor, and the 
 the intersection  
 $D_0\cap \ldots \cap D_d$ is empty.

Fix another  abelian variety  $F$ with the same dimension as $Y$.  We are going to produce a smooth foliation $\F$ on $X=Y\times F$  such that the image of $\pi_*T\F$ in $TY$ is isomorphic to $TY(\log D)$; in particular $\phi\colon \pi_*T\F \to TY$ is injective and generically surjective, but its image is not the tangent sheaf of a foliation since $TY/TY(\log D)$ is not torsion-free.

\begin{lemma}
\label{omega(logD) gg}
$\Omega_Y^1(\log D)$ is globally generated.
\end{lemma}
\begin{proof}
Let $p\in Y$. Then $p$ belongs to at most $d$ of the $D_i$-s, let's say $D_1,\ldots, D_k$. Let us fix coordinates $x_1,\ldots x_d$ of $Y$ around $p$ such that $D_i=\{x_i=0\}$ for $i=1,\ldots, k$; then the local sections of $\Omega_Y^1(\log D)$ around $p$ are generated by the forms
$$\frac {dx_1}{x_1},\ldots , \frac{dx_k}{x_k},dx_{k+1},\ldots, dx_d.$$

Now if $l_0,\ldots l_d$ are linear equations on $\mb P^N$ defining the $D_i$-s, then the restriction to $Y$ of the meromorphic form $d\log (l_i/l_j)$ defines a global section of $\Omega_Y^1(\log D)$ whose only poles are along $D_i$ and $D_j$.
In local coordinates around $p$ we have therefore
$$d\log (l_i/l_0)=\lambda_i \frac{dx_i}{x_i}+hol$$
for $i=1,\ldots, k$. Since $Y$ is a torus, the sheaf of holomorphic forms $\Omega_Y^1$ is trivial and in particular globally generated. This proves that each one of the forms $dx_i/x_i$ can be generated by global sections, which proves the statement.
\end{proof}

Any section $\sigma: Y \to X$ of $\pi:X \to Y$ defines a foliation $\mathcal H$ on $X$ with leaves  equal to translations of $\sigma(Y)$ by elements of $F$.  The foliation $\mathcal H$ induces a  natural inclusion of
$\pi^* TY(\log D)$ in $TX$, namely
\[
h \colon \pi^* TY(\log D) \to T \mathcal H \subset TX .
\]
In this way we can see $\pi^* TY(\log D)$ as an involutive subsheaf of $TX$. Note that it does  not
coincide with the
tangent sheaf of $\mathcal H$; the cokernel of the inclusion $\pi^* TY(\log D)\to T\mathcal H$ is not torsion-free.

We want to construct a {\it vertical} perturbation
$$v\colon \pi^* TY(\log D) \to TX/Y \hookrightarrow TX$$
in order to define the  sheaf $T\F$  as the  image of the sum morphism
$$h+v\colon \pi^ * TY(\log D)\to TX.$$

Let $y_1,\ldots , y_l$ be linear coordinates on the universal covering of  $F$. The morphism $v$ can be written as
$$v=\sum_{i=1}^l (\pi^* s_i)\frac{\partial}{\partial y_i}$$
for some global sections $s_1\ldots ,s_l$ of $\Omega_Y^1(\log D)$; reciprocally, every choice of global sections $s_1,\ldots ,s_l$ gives a morphism $v\colon TY(\log D)\to \pi_*TX/Y$.
We choose global sections $s_1,\ldots ,s_l$ such that on each point of $D$ the image of the vertical morphism defined by the $s_i$ is a vector space of maximal dimension $d$; the existence of such sections for a suitable $l$ is assured by Lemma \ref{omega(logD) gg}.

In order for $T\F$ to be the tangent sheaf of a smooth foliation, we have to verify that: (a) $T\mathcal F$ is involutive; $T\F$ is locally free;
and $TX/T\F$ is locally free. The fact that $T\F$ and $TX/T\F$ are locally free follows from the definition of $v$.

In order to prove the integrability of $T\mathcal F$ take two local sections of $\pi^* TY(\log D)$
$$v_1=\sum_{i=1}^d a_i\de_{x_i},\qquad v_2=\sum_{j=1}^d b_j\de_{x_j},$$
where $x_1,\ldots x_d$ are local coordinates on $Y$ and $\de_{x_i}:=\de/\de x_i$.
We have to check that
$$[(h+v)(v_1);(h+v)(v_2)]\in \im (h+v).$$
Since $\pi^* TY(\log D)$ is globally generated around fibers of $\pi$, we can assume without loss of generality that  $a_i$ and $b_j$ only depend on the $x$ coordinates (i.e. are constant along the fibers of $\pi$). Thus
$$[(h+v)(v_1);(h+v)(v_2)]=h([v_1;v_2])+[h(v_1);v(v_2)]+[v(v_1);h(v_2)].$$
Hence we have to prove that $[h(v_1);v(v_2)]+[v(v_1);h(v_2)]=v([v_1,v_2])$. We have
$$[v_1;v_2]=\sum a_i\frac {\de b_j}{\de x_i} \de_{x_j}-\sum b_j \frac{\de a_i}{\de x_j}\de_{x_i}=\sum \left( a_i\frac{\de b_j}{\de x_i} - b_i\frac {\de a_j}{\de x_i} \right)\de_{x_j},$$
so that
$$v([v_1;v_2])=\sum \left( a_i\frac{\de b_j}{\de x_i} - b_i\frac {\de a_j}{\de x_i} \right)s_k(\de_{x_j})\de_{y_k}.$$
On the other hand
$$[h(v_1);v(v_2)]=\sum a_i\frac {\de b_j}{\de x_i}s_k(\de_{x_j})\de_{y_k}+\sum a_ib_j\frac {\de}{\de x_i}(s_k(\de_{x_j}))\de_{y_k}$$
so that
$$[h(v_1);v(v_2)]+[v(v_1);h(v_2)]=$$
$$v([v_1;v_2])+\sum a_ib_j\left( \frac {\de}{\de x_j}(s_k(\de_{x_i}))- \frac {\de}{\de x_i}(s_k(\de_{x_j})) \right) \de_{y_k}.$$

Since the equality must be true for all choices of $v_1$ and $v_2$, we now have to prove that for all $i,j,k$
$$\frac {\de}{\de x_j}(s_k(\de_{x_i}))- \frac {\de}{\de x_i}(s_k(\de_{x_j}))=0,$$
that is
$ds_k(\de_{x_i},\de_{x_j})=0$.

Since $s_k$ is a logarithmic $1$-form with normal crossing polar divisor on a compact K\"ahler manifold, it must be  closed  by  a Theorem of Deligne, see  \cite[Corollary 3.2.14]{MR0498551} for the original proof, and \cite{MR1346909} or \cite[Lemma 2.1]{MR1800822} for short analytic proofs.

\bibliography{references}{}
\bibliographystyle{plain}
\end{document}